\documentclass[12pt]{amsart}
\usepackage{amssymb,latexsym,amsthm}
\usepackage{amsmath,amsfonts,amssymb}
\usepackage{graphicx}
\usepackage{color}
\usepackage{mathrsfs}

\newcommand{\la}{\lambda}

\newcommand{\Bb}{\mathcal{B}}
\newcommand{\Bbt}{\widetilde{\mathcal{B}}}

\newcommand{\ma}{\mathcal{A}}
\newcommand{\coa}{C(X_1,\ma_1)}
\newcommand{\cta}{C(X_2,\ma_2)}
\newcommand{\ps}{\operatorname{PS}}

\newcommand{\lxm}{\Lip(X,\mathcal{A})}

\newcommand{\lxom}{\Lip(X_1,\mathcal{A}_1)}
\newcommand{\lxtm}{\Lip(X_2,\mathcal{A}_2)}
\newcommand{\Lip}{\operatorname{Lip}}

\newcommand{\rdy}{\rho \circ \delta_{y}}
\newcommand{\yr}{Y_{\rho}}

\newcommand{\id}{\operatorname{Id}}

\newtheorem{theorem}{Theorem}[section]
\newtheorem{lemma}[theorem]{Lemma}
\newtheorem{cor}[theorem]{Corollary}
\newtheorem{prop}[theorem]{Proposition}

\theoremstyle{definition}

\newtheorem{example}[theorem]{Example}

\theoremstyle{remark}
\newtheorem{remark}[theorem]{Remark}

\begin{document}

\author{
Shiho~Oi
}
\address{
Department of Mathematics, Faculty of Science, Niigata University, Niigata
950-2181 Japan.
}
\email{shiho-oi@math.sc.niigata-u.ac.jp
}

\title[]
{Jordan $*$-homomorphisms on the spaces of continuous maps taking values in $C^{*}$-algebras 
}




\begin{abstract} 
Let $\ma$ be a unital $C^{*}$-algebra. We consider Jordan $*$-homomorphisms on $C(X, \ma)$ and Jordan $*$-homomorphisms on $\Lip(X,\ma)$. 
More precisely,  for any unital $C^{*}$-algebra $\ma$, we prove that  every Jordan $*$-homomorphism on $C(X,\ma)$ and every Jordan $*$-homomorphism on $\lxm$ is represented as a weighted composition operator by using the irreducible representations of $\ma$. In addition, when $\ma_1$ and $\ma_2$ are primitive $C^{*}$-algebras, we characterize the Jordan $*$-isomorphisms. 
These results unify and enrich previous works on algebra $*$-homomorphisms on $C(X, \ma)$ and $\Lip(X,\ma)$ for several concrete examples of $\ma$. 

\end{abstract}
\maketitle
\section{Introduction}
For any elements $a, b$ in a $*$-algebra, their Jordan product is defined by $a \circ b=\frac{ab+ba}{2}$. Let $\mathcal{B}_1$ and $\mathcal{B}_2$ be  $*$-algebras.   A Jordan $*$-homomorphism from $\mathcal{B}_1$ into $\mathcal{B}_2$ is a linear map $J:\mathcal{B}_1 \to \mathcal{B}_2$ such that $J(a \circ b)=J(a) \circ J(b)$ for any $a,b \in \mathcal{B}_1$ ( or equivalently, $J(a^2)=J(a)^2$ for any $a \in \mathcal{B}_1$) 
and $J(a)^{*}=J(a^{*})$ for any $a \in \mathcal{B}_1$. 
A well-known theorem of Kadison in \cite{Kad} tells that every unital surjective linear isometry between unital $C^{*}$-algebras is a Jordan $*$-isomorphism. Accordingly, Jordan $*$-homomorphisms are quite important and basic to an investigation. They form the foundation of $*$-algebra theories.  
Thus the study of Jordan $*$-homomorphisms between $*$-algebras has caused a great deal of interest. 

We now consider  $*$- algebras of vector-valued continuous functions. We 
emphasize that  on the one hand, the algebra $*$- homomorphisms on $*$- algebras of vector-valued continuous functions are well studied, on the other hand there are few results on Jordan $*$-homomorphisms on $*$- algebras of vector-valued continuous functions. 
The purpose of this paper is to investigate Jordan $*$-homomorphisms on two $*$-algebras of vector-valued functions. The first $*$-algebra  is $C(X, \ma)$ which is the space of all continuous maps from a compact Hausdorff space $X$ into a unital $C^{*}$-algebra $\ma$. More precisely, let $X$ be a compact Hausdorff space and $\ma$ be a unital $C^{*}$-algebra with the norm $\|\cdot\|_{\ma}$. Then
\[
C(X,\ma)=\{ F: X \to \ma \ | \ F\text{ is  continuous} \}.
\]
 It is easy to see that $C(X, \ma)$ is a unital $C^{*}$-algebra with the norm $\|F\|=\sup_{x \in X}\|F(x)\|_{\ma}$ for any $F \in C(X, \ma)$. The second $*$-algebra is $\Lip(X,\ma)$, which is the space of all Lipschitz maps from a compact metric space $X$ into a unital $C^{*}$-algebra $\ma$. More precisely, let $X$ be a compact metric space and $\ma$ be a unital $C^{*}$-algebra. Then
\[
\lxm=\{ F: X \to \ma \ | \ L(F):=\sup_{x \neq y \in X}\frac{\|F(x)-F(y)\|_{\ma}}{d(x,y)} < \infty \}.
\]
With $\|F\|_{L}=\sup_{x \in X}\|F(x)\|_{\ma}+L(F)$ for any $F \in \lxm$, $\lxm$ is a Banach algebra. Since we define the involution $*$ on $\lxm$ by $F^{*}(x)=[F(x)]^{*}$ for all $x \in X$,  $\lxm$ is a Banach $*$-algebra. 
Throughout the paper, $\ma$, $\ma_1$ and $\ma_2$ will be unital $C^{*}$-algebras. 

In \cite[Theorem 12]{LT}, Leung and Tang proved the following. For $i=1$ or $2$, let $X_i$ be realcompact spaces. 
Suppose that $T: C(X_1, \ma_1) \to C(X_2, \ma_2)$ is an algebra $*$-isomorphism so that $Z(F) \neq \emptyset \Leftrightarrow Z(TF) \neq \emptyset$, where $Z(F)=\{x \in X_i | F(x)=0 \}$ for each $F\in C(X_i,\ma_i)$ for $i=1,2$. Then there is a homeomorphism $\varphi: X_2 \to X_1$ and an algebra $*$-isomorphism $V_{y}: \ma_1 \to \ma_2$  for each $y \in X_2$ such that 
\begin{equation}\label{www}
TF(y)=V_y(F(\varphi(y))), \quad F \in C(X_1, \ma_1) \text{ and all } y \in X_2.
\end{equation}
Let us recall that an algebra $*$-homomorphism is a multiplicative linear operator that preserves the involution $*$.  It is natural to ask whether a Jordan $*$-isomorphism from $C(X_1, \ma_1)$ onto $C(X_2,\ma_2)$ 
can be represented as a weighted composition operator, which has a similar form to (\ref{www}). 
In this paper, we prove that every  Jordan $*$-isomorphism from $C(X_1, \ma_1)$ onto $C(X_2,\ma_2)$ is represented as a weighted composition operator by using the irreducible representations of $\ma_2$. In addition we characterize Jordan $*$-isomorphisms from $C(X_1, \ma_1)$ onto $C(X_2,\ma_2)$ if $\ma_1$ and $\ma_2$ are primitive $C^{*}$-algebras. 

We now introduce the study of  Jordan $*$-homomorphisms on $\lxm$.  As pointed out by Botelho and Jamison in \cite{BJb}, it is not known that whether algebra $*$-homomorphisms 
on $\lxm$ are bounded with $\|\cdot\|_{L}$ or not. Therefore there are some results for continuous algebra $*$-homomorphisms on $\lxm$. For example, a result of Botelho and Jamison \cite{BJb} says that if $T: \Lip(X_1, B(H)) \to \Lip(X_2, B(H))$ is a ps-continuous algebra $*$-homomorphism that preserves constant functions with $T(1)=1$, then there exist a Lipschitz function $\varphi: X_2 \to X_1$ and an algebra $*$-homomorphism $V: B(H) \to B(H)$ such that $T(F)(y)=V(F(\varphi(y)))$ for all $F \in \Lip(X_1, B(H))$ and $y \in X_2$. 
In \cite{BJm}, Botelho and Jamison also proved that if $T: \Lip(X_1, M_{n}(\mathbb{C})) \to \Lip(X_2, M_{n}(\mathbb{C}))$ is an algebra $*$-homomorphism and $T(1)=1$ then there exist an algebra $*$-homomorphism $V_{y}: M_{n}(\mathbb{C}) \to M_{n}(\mathbb{C})$ for each $y \in X_2$ and a continuous function $\varphi:X_2 \to X_1$ such that $T(F)(y)=V_y(F(\varphi(y)))$ for all $F \in \Lip(X_1, M_{n}(\mathbb{C}))$ and $y \in X_2$. As in the case of commutative $C^{*}$-algebras, in \cite{Oi} the author proved the following. Suppose that $X_2$ is connected. Let $M_i$ be compact Hausdorff spaces for $i=1,2$.  If $T: \Lip(X_1, C(M_1)) \to \Lip(X_2, C(M_2))$ is an algebra homomorphism with $T(1)=1$ then there exist a continuous map $\tau: M_2 \to M_1$, and a Lipschitz map $\varphi_{\phi}:X_2 \to X_1$ for each $\phi \in M_2$ 
 such that $T(F)(y)(\phi)=F(\varphi_{\phi}(y))(\tau(\phi))$ holds for any $F \in \Lip(X_1, C(M_1))$, $\phi \in M_2$ and $y \in X_2$. 
The above results tell us that there are a variety of forms of algebra $*$-homomorphisms.  We  see that  forms of  algebra $*$-homomorphisms and Jordan $*$-homomorphisms depend on the classes of $C^{*}$-algebras $\ma$. In Section \ref{LC}, we consider Jordan $*$-homomorphisms on $\lxm$ for any unital $C^{*}$-algebras $\ma$. We obtain a characterization of the Jordan $*$-isomorphisms from $\lxom$ onto $\lxtm$ when $\ma_1$ and $\ma_2$ are primitive $C^{*}$-algebras. This is an generalization of the above results by Botelho and Jamison in \cite{BJb,BJm}. 

To state our theorem, we need some notations. We say a state $\rho$ on $\ma$ is pure if it has the property that whenever $\tau$ is a positive linear functional on $\ma$ such that $\tau \le \rho$, there is a number $t \in [0,1]$ such that $\tau=t\rho$. We denote the set of all pure states on $\ma$ by $\ps(A)$.  
If $\rho$ is a state on $\ma$, then there is a cyclic representation $(\pi_{\rho}, H_{\rho})$ of $\ma$ with a unit cyclic vector $x_{\rho}$ such that $\rho(a)= <\pi_{\rho}(a)x_{\rho}, x_{\rho}>$ for all $a \in \ma$.
In this paper,  we denote the identity operator on  a set $B$ by $\id_{B}$. 
 We also denote the unity of a unital Banach algebra $B$ by $1_{B}$.  If there is no cause for confusion, we write just $1$.

\section{Jordan $*$-homomorphisms on $C(X,\ma)$}\label{RC}

In this section we state the theorems for Jordan $*$-homomorphisms on $C(X,\ma)$. In general, every Jordan $*$-homomorphism between two unital $C^{*}$-algebras is contractive. More 
precisely, let $\ma_1$ and $\ma_2$ be unital $C^{*}$-algebras.  Suppose that  $J: \ma_1 \to \ma_2$ is a Jordan $*$-homomorphism.  Since $J(1)$ is a projection in $\ma_2$, we have  $\|J\| \le 1$ by \cite[Corollary 1]{RD}.

We first introduce some additional notations and terminology. 
In this section, $X$,  $X_1$ and $X_2$ are compact Hausdorff spaces. 
For a unital $C^{*}$-algebra $\ma$, if its zero ideal is primitive, we call it a primitive $C^{*}$-algebra. 
Let $f \in C(X)$ and $a \in \ma$, where $C(X)$ is a commutative $C^{*}$-algebra of all complex valued continuous functions on $X$. We define the algebraic tensor product $f \otimes a: X \to \ma$ by
\[
(f\otimes a)(x)=f(x)a, \quad x \in X.
\]
We have $f \otimes a \in C(X,\ma)$ such that $\|f \otimes a\|=\|f\|_{\infty}\|a\|_{\ma}$. We denote the algebraic tensor product space by $C(X) \otimes \ma$.

For any unital $C^{*}$-algebra $\ma$, the following proposition is well-known. 
\begin{prop}\label{Ci}
If  $\mathcal{B}$ is a $C^{*}$-subalgebra of $\ma$ that contains the identity, every pure state on $\mathcal{B}$ is extend to a pure state on $\ma$.
\end{prop}
We see a proof of Proposition \ref{Ci} e.g. in \cite[Corollary 32.8]{Con}. On the other hand, 
we also have the next proposition. 
\begin{prop}\label{NC}
If $\Bb$ is a  $C^{*}$-subalgebra of $\ma$, every pure state on $\mathcal{B}$ can be extend to a pure state on $\ma$.
\end{prop}
Proposition \ref{NC} means that  even if $\mathcal{B}$ does not contain the identity of $\ma$, each pure state on $\mathcal{B}$ can be extend.
Unfortunately we could not find a proof of Proposition \ref{NC} in the literature, so we present proofs in detail for completeness. To prove Proposition \ref{NC}, we apply Proposition \ref{Ci}. 

\begin{proof}
Let $\psi$ be a pure state on $\Bb$. If $\Bb$ has the identity of $\ma$, by Proposition \ref{Ci} there exists a pure state $\phi$ on $\ma$ such that $\phi|_{\Bb}= \psi$. Thus suppose that $\Bb$ does not have the identity. Let $\Bbt$ be the unitization of $\Bb$. We define a linear functional $\psi_1$ on $\Bbt$ by $\psi_1(b+ \beta)=\psi(b)+\beta$ for  $b \in \Bb$, $\beta \in \mathbb{C}$. Then for any positive element $(b+\beta)^{*}(b+\beta) \in \Bbt$, we have $\psi_1((b+\beta)^{*}(b+\beta)) \ge (|\psi(b)|-|\beta|)^2+2(|\overline{\psi(b)}\beta|+ Re(\overline{\psi(b)}\beta)) \ge 0$ and $\psi_{1}(1)=1$. Thus $\psi_1$ is a state on $\Bbt$. Let $\tau$ be a positive linear functional on $\Bbt$  such that $\tau \le \psi_1$. For any $b \in \Bb$, we get $\tau(b^{*}b) \le \psi_1(b^{*}b) = \psi(b^{*}b)$. Thus $\tau|_{\Bb}$ is a positive linear functional on $\Bb$ such that $\tau|_{\Bb} \le \psi$. Since $\psi$ is a pure state on $\Bb$, there is $t \in [0,1]$ such that $\tau|_{\Bb} =t \psi$. Let $(u_{\la})_{\la \in \Lambda}$ be an approximate unit for $\Bb$. As $\psi(u_{\la})^2 \le \psi(u_{\la}^2) \le 1$, and $\lim_{\la} \psi(u_{\la})=\|\psi\|=1$, we get $\lim_{\la}\psi(u_{\la}^{2})=1$. As $\tau \le \psi_1$, for any $\beta \in \mathbb{R}$ we have 
\begin{multline*}
0 \le (\psi_1-\tau)((u_{\la}+\beta)^{*}(u_{\la}+\beta))=(\psi_1-\tau)(u_{\la}^{2}+2\beta u_{\la}+\beta^{2})\\
=(1-t)(\psi(u_{\la}^{2}))+(1-t)(2\psi(u_{\la})\beta)+(1-\|\tau\|)\beta^{2}.
\end{multline*} 
Thus $0 \le (1-t)+2(1-t)\beta+(1-||\tau\|)\beta^{2}$ for any $\beta \in \mathbb{R}$. This implies that $(1-t)^2-(1-t)(1-\|\tau\|)=(1-t)(\|\tau\|-t) \le 0$. As $t \in [0,1]$, we obtain $\|\tau\|\le t$. 
For any $b \in \Bb$, we have $\|t\psi(b)\|=\|\tau(b+0)\|\le \|\tau\| \|b\|$. As $\|\psi\|=1$ we get $t \le \|\tau\|$. Thus $\|\tau\|=t$. Then for any $b \in \Bb$ and $\beta \in \mathbb{C}$, we get  $\tau(b+ \beta)=\tau(b)+\beta\tau(1)=t\psi(b)+\beta t=t(\psi(b)+\beta)=t\psi_1(b+\beta)$. Hence $\tau=t \psi_1$ and $\psi_1$ is a pure state on $\Bbt$. Since $\Bbt$ contains the identity of $\widetilde{\ma}$, by Proposition \ref{Ci}, there is a pure state $\psi_1'$ on  $\widetilde{\ma}$ such that $\psi_1'|_{\Bbt}=\psi_1$. We shall show that $\psi_{1}'|_{\ma}$ is a pure state on $\ma$. Let $\tau$ be a positive linear functional such that $\tau \le \psi_1'|_{\ma}$. Then we define a positive linear functional $\tau'$ on  $\widetilde{\ma}$ by $\tau'(a+ \beta)=\tau(a)+\beta \|\tau\|$ for any $a+\beta \in  \widetilde{\ma}$. Then $\|\tau'\|=\|\tau\|$. For any positive element $(a+\beta)^*(a+ \beta) \in  \widetilde{\ma}$, we have $(\psi_1'-\tau')((a+\beta)^*(a+ \beta))=(\psi_1'|_{\ma}-\tau)((a+\beta 1_{\ma})^*(a+ \beta1_{\ma})) \ge 0$. Hence $\psi_1' \ge \tau'$. Since $\psi_1'$ is a pure state on $\widetilde{\ma}$, there exists $t \in [0,1]$ such that $\tau'=t \psi_1'$. For any $a \in \ma$ we have $\tau(a)=\tau'(a)=t\psi_1'(a)=t\psi_1'|_{\ma}(a)$. Thus $\psi_1'|_{\ma}$ is a pure state on $\ma$. We get $\psi(b)=\psi_1(b)=\psi_1'|_{\Bbt}(b)=\psi_1'|_{\ma}(b)$ for any $b \in \Bb$. Therefore $\psi_1'|_{\ma}=\psi$ on $\Bb$.
\end{proof}

For $y \in X$, let $\delta_{y}$ denote the evaluation function on $C(X, \ma)$ defined by $\delta_{y}(F)=F(y)$. We now describe the pure states on a $C^{*}$-subalgebra of $C(X, \ma)$. 

\begin{prop}\label{333}
Let $\mathcal{B}$ be a $C^{*}$-subalgebra of $C(X, \ma)$. Then $\ps(\mathcal{B}) \subset \{ \rdy |\  \rho \in \ps(\ma), y \in X\}$.
\end{prop}

\begin{proof}
Let $p \in \ps(\mathcal{B})$. By Proposition \ref{NC}, there exists a pure state $\phi$ on $C(X,\ma)$ such that $\phi|_{\mathcal{B}}=p$.  Suppose that there are $\tau_1, \tau_2$ in $C(X,\ma)^{*}_1$, which is the closed unit ball of the dual of $C(X, \ma)$, such that $\phi=\frac{\tau_1+\tau_2}{2}$. Then we have $1=\phi(1)=\frac{\tau_1(1)+\tau_2(1)}{2}$, we have $\tau_1(1)=\tau_2(1)=1$. Thus $\tau_1$ and $\tau_2$ are states on $C(X, \ma)$. As $\phi$ is a pure state, we get $\phi=\tau_1=\tau_2$. Therefore $\phi$ is an extreme point of $C(X, \ma)^{*}_1$. 
By Corollary 2.3.6. in \cite{FJB}, there exist an extreme point $\rho$ of the closed unit ball of the dual of $\ma$ and $y \in X$ such that $\phi=\rho \circ \delta_y$. As $1=\phi(1)=\rho \circ \delta_y(1)=\rho(1)$, $\rho$ is a pure state on $\ma$.
\end{proof}

\begin{theorem}\label{T1}
Let $J$ be a Jordan $*$-homomorphim of $\coa$ into $\cta$. Let $\mathcal{B}$ be the $C^{*}$-algebra generated by $J(\coa)$. For any $\rdy \in \ps(\mathcal{B})$, let $Y_{\rho}=\{ x \in X_2 | \  \rho \circ \delta_{x} \in  \ps(\mathcal{B})\}$. Then there exist a continuous map $\varphi_{\rho}: \yr \to X_1$ and a Jordan $*$-homomorphism $V_y : \ma_1 \to \ma_2$ for each $y \in \yr$ such that 
\[
\pi_{\rho}(JF(y))=\pi_{\rho}(V_y(F(\varphi_{\rho}(y))))
\]
for all $F \in \coa$ and all $y \in \yr$.
\end{theorem}
We note that as $\rdy \in \ps(\mathcal{B})$, $y \in Y_{\rho} \neq \emptyset$.

\begin{proof}
By Proposition \ref{333}, fix $\rdy \in \ps(\mathcal{B})$. Then $\pi_{\rdy}$ is an irreducible representation of $\Bb$. Thus $\pi_{\rdy}(\Bb)$ acts irreducibly on $H_{\rdy}$. Then $\pi_{\rdy}(\Bb)'=\mathbb{C}\id_{H_{\rdy}}$, where $\pi_{\rdy}(\Bb)'$ is the commutant of $\pi_{\rdy}(\Bb)$ in $B(H_{\rdy})$. By Corollary 3.4. in \cite{Stormer}, $\pi_{\rdy} \circ J:\coa \to B(H_{\rdy})$ is either a $*$-homomorphism or an anti $*$-homomorphism. For any $f \in C(X_1)$, if it is a homomorphism, we have 
\begin{equation*}
\begin{split}
\pi_{\rdy}(J(F))\pi_{\rdy}(J(f \otimes 1))=\pi_{\rdy}(J(F\cdot f\otimes 1))\\=\pi_{\rdy}(J(f \otimes 1 \cdot F))=\pi_{\rdy}(J(f \otimes 1))\pi_{\rdy}(JF)
\end{split}
\end{equation*}
for all $F \in \coa$. Thus $J(f \otimes 1) \in \pi_{\rdy}(\Bb)'=\mathbb{C}\id_{H_{\rdy}}$. If $\pi_{\rdy} \circ J$ is an anti $*$-homomorphism, we obtain $J(f \otimes 1) \in \pi_{\rdy}(\Bb)'=\mathbb{C}\id_{H_{\rdy}}$ similarly. There is $\lambda_{f} \in \mathbb{C}$ such that 
\begin{equation}\label{1}
\pi_{\rdy} \circ J(f \otimes 1)=\lambda_{f} \cdot \id_{H_{\rdy}}.
\end{equation}
If  $f=1 \in C(X_1)$, there is $\la_1 \in \mathbb{C}$ such that $\pi_{\rdy} \circ J(1)=\lambda_{1} \cdot \id_{H_{\rdy}}$. We have $\lambda_{1} \cdot \id_{H_{\rdy}}=\pi_{\rdy} \circ J(1)=\pi_{\rdy} \circ J(1^2)=(\pi_{\rdy} \circ J(1))(\pi_{\rdy} \circ J(1))=\lambda_{1}^2 \cdot \id_{H_{\rdy}}$. This implies that $\lambda_{1}=1$ or $\lambda_{1}$=0. Suppose that  $\la_1=0$. Then we get $\pi_{\rdy}(J(1))=0$. As $\rdy$ is a pure state on $\Bb$, there exists $F \in \coa$ such that $\rdy(JF) \neq 0$. Therefore $\pi_{\rdy}(JF) \neq 0$. Then we have $0 \neq \pi_{\rdy}(JF)=\pi_{\rdy}(J(F \cdot 1))=\pi_{\rdy}(JF) \pi_{\rdy}(J1)=0$. This is a contradiction. Hence $\la_1 =1$. Then $\pi_{\rdy} \circ J(1)=\id_{H_{\rdy}}$ and by (\ref{1}) we get  $\pi_{\rdy} \circ J(f \otimes 1)=\lambda_{f} \cdot \id_{H_{\rdy}}=\pi_{\rdy} \circ J(\lambda_{f} \cdot 1)$. So we get $\pi_{\rdy}(J(f\otimes 1)-J(\lambda_f 1))=0$. 
This implies $\rdy((J(f\otimes 1)-J(\lambda_f  1))^*(J(f\otimes 1)-J(\lambda_f  1)))=0$. Therefore we get
\begin{equation}\label{1.5}
J^{*}(\rdy)(f \otimes 1-\lambda_f \cdot 1)=0
\end{equation} 
and 
\begin{equation}\label{2}
\pi_{\rho}(J(f\otimes 1)(y))=\lambda_{f}(\pi_{\rho}(J(1)(y)))
\end{equation}
 for any $f \in C(X_1)$.  
Since $\rdy \in \ps(\Bb)$, we have $J^{*}(\rdy) \in \ps(\coa)$. There exist $\phi_{\rho,y} \in \ps(\ma_1)$ and $x \in X_1$ such that $J^{*}(\rdy)=\phi_{\rho,y} \circ \delta_{x}$. 
Now we define $\varphi_{\rho}:Y_{\rho} \to X_1$ by $J^{*}(\rdy)=\phi_{\rho,y} \circ \delta_{\varphi_{\rho}(y)}$. Then by (\ref{1.5}) and $\phi_{\rho,y} (1)=1$, we obtain 
$f(\varphi_{\rho}(y))=\lambda_f$. By (\ref{2}), we have
\begin{equation}\label{3}
\pi_{\rho}(J(f \otimes 1)(y))=f(\varphi_{\rho}(y))\pi_{\rho}(J(1)(y))
\end{equation} 
for any $f \in C(X_1)$ and $y \in Y_{\rho}$. 

Fix $y \in Y_{\rho}$. We define $V_{y}: \ma_1 \to \ma_2$ by 
\[
V_y(a)=J(1 \otimes a)(y)
\]
for any $a \in \ma_1$. It is clear that $V_y$ is a Jordan $*$-homomorphism.
Let $\Bb_y$ be the $C^{*}$-algebra generated by $J(\coa)(y)$. As $\rdy \in \ps(\Bb)$, we shall show that $\rho$ is a pure state on $\Bb_y$. As $\rdy \in \ps(\Bb)$, for any $\epsilon>0$, there is $G \in \Bb$ with $\|G\|_{\infty} \le 1$ such that $|\rdy(G)| \ge 1-\epsilon$. This implies that $|\rho(G(y))| \ge 1-\epsilon$. Since $\|G(y)\| \le \|G\|_{\infty} \le 1$, we get $\|\rho|_{\Bb_y}\|\ge 1$. On the other hand, as $\rho \in \ps(\ma_2)$, $\rho$ is a state on $\Bb_y$. Let $\tau$ be a positive linear functional such that $\tau \le \rho$ on $\Bb_y$. Then for any positive element $F \in \coa$, $\tau(JF(y)) \le \rho(JF(y))$. This implies that $\tau \circ \delta_y \le \rdy$ on $\Bb$. Since $\rdy$ is a pure state on $\Bb$, there is  $t \in [0,1]$ such that $\tau \circ \delta_y = t \rho \circ \delta_y$. For any $F \in \coa$, we get $\tau(JF(y))=t\rho(JF(y))$. So $\tau=t\rho$ on $\Bb_y$. Thus $\rho$ is a pure state on $\Bb_y$. Since a map from $\coa$ into $\Bb_y$ by $F \mapsto J(F)(y)$ is a Jordan $*$-homomorphism, the map $F \mapsto \pi_{\rho}(J(F)(y))$ is either a homomorphism or an anti-homomorphism. For any $f \in C(X_1)$ and $a \in \ma_1$, if it is a homomorphism, by (\ref{3}) we have
\begin{multline*}
\pi_{\rho}(J(f \otimes a)(y))=\pi_{\rho}(J(f\otimes 1)(y))\pi_{\rho}(J(1 \otimes a)(y))\\=f(\varphi_{\rho}(y))\pi_{\rho}(J(1)(y))\pi_{\rho}(J(1 \otimes a)(y))
=f(\varphi_{\rho}(y))\pi_{\rho}(V_y(a))\\=\pi_{\rho}(V_y(f(\varphi_{\rho}(y))a))=\pi_{\rho}(V_y(f \otimes a(\varphi_{\rho}(y)))).
\end{multline*}
If it is an anti $*$-homomorphism,  for  any $f \in C(X_1)$ and $a \in \ma_1$, $\pi_{\rho}(J(f \otimes a)(y))=\pi_{\rho}(V_y(f \otimes a(\varphi_{\rho}(y))))$ holds similarly as above. This implies that $\pi_{\rho}(J(F)(y))=\pi_{\rho}(V_y(F(\varphi_{\rho}(y))))$ holds for any $F \in C(X_1) \otimes \ma_1$. Since $C(X_1) \otimes \ma_1 $ is dense in $C(X_1,\ma_1)$, for any $F \in \coa$, there exists $G_n \in C(X_1) \otimes \ma_1$ such that $\|F-G_n\|_{\infty} \to 0$ as $n \to \infty$. Since every Jordan $*$-homomorphism between $C^{*}$-algebras is  contractive,  we have
\begin{equation*}
\begin{split}
&\|\pi_{\rho}(JF(y))-\pi_{\rho}(V_y(F(\varphi_{\rho}(y)))\|_{B(H_{\rho})}\\
&\le \|\pi_{\rho}(JF(y))-\pi_{\rho}(JG_n(y))\|_{B(H_{\rho})}+\|\pi_{\rho}(V_y(G_n(\varphi_{\rho}(y))))-\pi_{\rho}(V_y(F(\varphi_{\rho}(y)))\|_{B(H_{\rho})}\\
&\le \|JF(y)-JG_n(y)\|_{\ma_2}+\|V_y(G_n(\varphi_{\rho}(y)))-V_y(F(\varphi_{\rho}(y))\|_{\ma_2}\\
&\le \|JF-JG_n\|_{\infty}+\|G_n(\varphi_{\rho}(y)))-F(\varphi_{\rho}(y))\|_{\ma_1}\\
&\le 2\|F-G_n\|_{\infty} \to 0
\end{split}
\end{equation*}
as $n \to \infty$. Thus for any $F \in \coa$ we get 
\[
\pi_{\rho}(JF(y))=\pi_{\rho}(V_y(F(\varphi_{\rho}(y)))), \quad y \in Y_{\rho}.
\]
Finally we shall show that $\varphi_{\rho}:Y_{\rho} \to X_1$ is continuous. Since $\rho \in \ps(\Bb_y)$ for any $y \in Y_{\rho}$, we have $\pi_{\rho}(J(1)(y))=\id_{H_{\rho}}$. By (\ref{3}), we get $\pi_{\rho}(J(f \otimes 1)(y))=f(\varphi_{\rho}(y))\id_{H_{\rho}}$ for any $f \in C(X_1)$ and $y \in Y_{\rho}$. Let $\{y_{\lambda}\} \subset Y_{\rho}$ be a net with $y_{\lambda} \to y_0 \in Y_{\rho}$. For any $f \in C(X_1)$ we have $\pi_{\rho}(J(f \otimes 1)(y_\lambda)) \to \pi_{\rho}(J(f\otimes 1)(y_0))$. Hence we get $f(\varphi_{\rho}(y_{\lambda})) \to f(\varphi_{\rho}(y_0))$ for any $f \in C(X_1)$. Let $U$ be a neighbourhood of $\varphi_{\rho}(y_0)$. By Uryzorn's lemma there exists $g \in C(X_1)$ such that $g=0$ on $X_1 \setminus U$ and $g(\varphi_{\rho}(y_0))=1$. Thus there is $\lambda_0$ such that $|g(\varphi_{\rho}(y_\lambda))| \ge \frac{1}{2}$ for all $\lambda \ge \lambda_0$. This implies that $\varphi_{\rho}(y_{\lambda}) \in U$ if $\lambda \ge \lambda_0$. We conclude that $\varphi_{\rho}$ is  continuous. We complete the proof.
\end{proof}

\begin{theorem}\label{T2}
Let $J:\coa \to \cta$ be a Jordan $*$-isomorphism. Then there exist  a continuous map $\varphi_{\rho}: X_2 \to X_1$ for any $\rho \in \ps(\ma_2)$ and  a Jordan $*$-homomorphism  $V_y: \ma_1 \to \ma_2$ for each $y \in X_2$ such that 
\[
\pi_{\rho}(JF(y))=\pi_{\rho}(V_y(F(\varphi_{\rho}(y))))
\]
for all $F \in \coa$ and all $y \in X_2$.
\end{theorem}

\begin{proof}
Since $\ps(\cta)=\{\rdy | \ \rho \in \ps(\ma_2), y \in X_2\}$, we get $Y_{\rho}=X_2$ for any $\rho \in \ps(\ma_2)$. Applying Theorem \ref{T1}, we obtain the theorem.
\end{proof}

The next example shows that the maps $\varphi_{\rho}: \yr \to X_1$ and $V_y : \ma_1 \to \ma_2$ are not always bijections even if $J$ is a Jordan $*$-isomorphism.
\begin{example}
Let  $J: C(\{a\},C(\{b,c\})) \to C(\{a,b\},C(\{c\}))$ be a Jordan $*$-isomorphism given by $JF(a)(c)=F(a)(b)$ and $JF(b)(c)=F(a)(c)$. Let $\delta_{c} \circ \delta_{a}$ be a linear functional on $ C(\{a\},C(\{b,c\}))$, which is defined by $\delta_{c} \circ \delta_{a}(F)=F(a)(c)$ for all $F \in  C(\{a\},C(\{b,c\}))$.  Then $\delta_{c} \circ \delta_{a}$ is a pure state on $C(\{a,b\},C(\{c\}))$. Then $V_a: C(\{b,c\}) \to C(\{c\})$ is a Jordan $*$-homomorphism given by $V_a(g)=J(1 \otimes g)(a)$ for all $g \in C(\{b,c\})$. But $V_a$ is not a Jordan $*$- isomorphism. In addition, $\varphi_{\delta_{c}}:\{a,b\} \to \{a\}$ is a continuous map given by $\varphi_{\delta_{c}}(x)=a$ for any $x \in \{a,b\}$. This is not a homeomorphism.
\end{example}

\begin{remark}
 For any $F \in C(X,\ma)$, we define the zero set of $F$ by $Z(F)=\{x \in X | \ F(x)=0\}$. In \cite[Theorem 12]{LT}, Leung and Tang confirmed a conjecture of Ercan and \"{O}nal in a general setting. As a result, they proved the following. Let $X_1$ and $X_2$ be realcompact spaces.  If  an algebra $*$-isomorphism from $\coa$ onto $\cta$ satisfies  $Z(F) \neq \emptyset \Leftrightarrow Z(TF) \neq \emptyset$ for any $F \in \coa$, then $X_1$ is homeomorphic to $X_2$ and $\ma_1$ is $C^{*}$-isomorphic to $\ma_2$.  Theorem \ref{T2} implies that  if $X_1$ and $X_2$ are compact Hausdorff spaces, every algebra $*$-isomorphism from $\coa$ onto $\cta$ is represented as a weighted composition operator without assumptions on  zero sets. However $X_1$ is not always homeomorphic to $X_2$ and $\ma_1$ is not always $C^{*}$-isomorphic to $\ma_2$. 
\end{remark}

On the other hand if $\ma_1$ and $\ma_2$ are primitive $C^{*}$-algebras,  we obtain that  $X_1$ is homeomorphic to $X_2$ and $\ma_1$ is $C^{*}$-isomorphic to $\ma_2$.

\begin{theorem}\label{btry}
Assume that $\ma_1$ and $\ma_2$ are primitive. Then $J: \coa \to \cta$ be a Jordan $*$-isomorphism if and only if there exist a homeomorphism $\varphi$ from $X_2$ onto $X_1$ and  a Jordan $*$-isomorphism $V_y: \ma_1 \to \ma_2$ for each $y \in X_2$ so that  the maps $y \mapsto V_y$ and $x \mapsto V^{-1}_{\varphi^{-1}(x)}$ are continuous with respect to the strong operator topology such that 
\begin{equation}\label{111}
JF(y)=V_y(F(\varphi(y)))
\end{equation}
for all $F \in \coa$ and $y \in X_2$.
\end{theorem}

\begin{proof}
Firstly, suppose that $J$ is of the form described as (\ref{111}) in the statement. We shall prove that $JF \in \cta$ for any $F \in \coa$. Let $\{y_{\lambda}\} \subset X_2$ be a net with $y_{\lambda} \to y_0$. Then 
\begin{equation*}
\begin{split}
&\|JF(y_\lambda)-JF(y)\|_{\ma_2}\\&=\|V_{y_{\lambda}}(F(\varphi(y_{\lambda})))-V_y(F(\varphi(y)))\|_{\ma_2} \\
&\le \|V_{y_{\lambda}}(F(\varphi(y_{\lambda})))-V_{y_{\lambda}}(F(\varphi(y)))\|_{\ma_2}+\|V_{y_{\lambda}}(F(\varphi(y)))-V_y(F(\varphi(y)))\|_{\ma_2}\\
&= \|F(\varphi(y_{\lambda}))-F(\varphi(y))\|_{\ma_2}+\|V_{y_{\lambda}}(F(\varphi(y)))-V_y(F(\varphi(y)))\|_{\ma_2} \to 0,
\end{split}
\end{equation*}
since the map $y \mapsto V_y$ is continuous with respect to the strong operator topology. For any $F,G \in \coa$, we have 
\begin{multline*}
J(F \circ G)(y)=V_{y}((F \circ G)(\varphi(y)))=V_{y}(F(\varphi(y)) \circ G(\varphi(y)))\\=V_{y}(F(\varphi(y))) \circ V_{y}( G(\varphi(y)))=JF(y) \circ JG(y)=(JF \circ JG)(y).
\end{multline*}
In addition we have $JF^{*}(y)=V_y(F^{*}(\varphi(y)))=(V_y(F(\varphi(y))))^{*}=(JF)^{*}(y)$ for any $F \in \coa$ and $y \in X_2$. Thus $J$ is a Jordan $*$- homomorphism. For  any $G \in \cta$, we define a map $F$ on $X_1$ by $F(x)=V^{-1}_{\varphi^{-1}(x)}(G(\varphi^{-1}(x)))$. Let $\{x_{\lambda}\} \subset X_1$ be a net with $x_{\lambda} \to x_0$. Then
\begin{equation*}
\begin{split}
&\|F(x_{\lambda})-F(x_0)\|_{\ma_1}\\&=\|V^{-1}_{\varphi^{-1}(x_{\lambda})}(G(\varphi^{-1}(x_{\lambda})))-V^{-1}_{\varphi^{-1}(x_0)}(G(\varphi^{-1}(x_0)))\|_{\ma_1}\\
&\le \|V^{-1}_{\varphi^{-1}(x_{\lambda})}(G(\varphi^{-1}(x_{\lambda})))-V^{-1}_{\varphi^{-1}(x_{\lambda})}(G(\varphi^{-1}(x_{0})))\|_{\ma_1} \\ & \quad \quad \quad \quad \quad  +\|V^{-1}_{\varphi^{-1}(x_{\lambda})}(G(\varphi^{-1}(x_{0})))-V^{-1}_{\varphi^{-1}(x_0)}(G(\varphi^{-1}(x_0)))\|_{\ma_1}\\
&= \|G(\varphi^{-1}(x_{\lambda}))-G(\varphi^{-1}(x_{0}))\|_{\ma_2}+\|V^{-1}_{\varphi^{-1}(x_{\lambda})}(G(\varphi^{-1}(x_{0})))-V^{-1}_{\varphi^{-1}(x_0)}(G(\varphi^{-1}(x_0)))\|_{\ma_1} \\ 
&\to 0,
\end{split}
\end{equation*}
as $x \mapsto V^{-1}_{\varphi^{-1}(x)}$ is continuous with respect to the strong operator topology. This implies that $F(x_{\lambda}) \to F(x_0)$ as $x_{\lambda} \to x_0$. Thus $F \in \coa$. Moreover for any $y \in X_2$ we have
\begin{equation*}
\begin{split}
J(F)(y)&=J(V^{-1}_{\varphi^{-1}(\cdot)}(G(\varphi^{-1}(\cdot))))(y)\\ &=V_y(V^{-1}_{\varphi^{-1}(\varphi(y))}(G(\varphi^{-1}(\varphi(y)))))=V_y(V^{-1}_y(G(y)))=G(y).
\end{split}
\end{equation*}
So $J$ is surjective. Suppose that $F \neq G \in \coa$, then there is  $x \in X_1$ such that $F(x) \neq G(x)$. Then there exists $y \in X_2$ such that $\varphi(y)=x$ and $V_y(F(\varphi(y))) \neq V_y(G(\varphi(y)))$. Thus $JF \neq JG$. So $J$ is injective. Thus $J$ is a Jordan $*$-isomorphism. 

We shall show that the converse holds. By Theorem \ref{T2}, there exist a continuous map $\varphi_{\rho}: X_2 \to X_1$ for each $\rho \in \ps(\ma_2)$ and  a Jordan $*$-homomorphism  $V_y: \ma_1 \to \ma_2$ for each $y \in X_2$ such that 
\[
\pi_{\rho}(JF(y))=\pi_{\rho}(V_y(F(\varphi_{\rho}(y))))
\]
for all $F \in \coa$ and $y \in X_2$. Since $\ma_2$ is a primitive $C^{*}$-algebra, there is a $\rho \in \ps(\ma_2)$ such that $\pi_{\rho}$ is faithful. We define $\varphi:X_2 \to X_1$ by $\varphi=\varphi_{\rho}$. Then we have
\[
JF(y)=V_y(F(\varphi(y))).
\]
We show that $\varphi$ is a homeomorphism. Let $y_1,y_2 \in X_2$ such that $\varphi(y_1) = \varphi(y_2)$. As $\ma_1$ is a primitive $C^{*}$- algebra, applying the similar argument to $J^{-1}$, there exist a continuous map $\psi:X_1 \to X_2$ and a Jordan $*$-homomorphism $S_x: \ma_2 \to \ma_1$ for each $x \in X_1$ such that $J^{-1}F(x)=S_x(F(\psi(x)))$, for all $F \in \cta$. Then we have 
\begin{equation}\label{sss}
F(x)=J^{-1}JF(x)=S_x(JF(\psi(x)))=S_x(V_{\psi(x)}F(\varphi(\psi(x))))
\end{equation}
for all $F \in \coa$. As $\coa$ separates the points of $X_1$, we get $\varphi(\psi(x))=x$ for any $x \in X_1$. 
 In addition we also have 
\begin{equation}\label{ss}
F(y)=JJ^{-1}F(y)=V_y(J^{-1}F(\varphi(y)))=V_y(S_{\varphi(y)}(F(\psi(\varphi(y)))))
\end{equation}
for all $F \in \cta$. Since $\cta$ separates the points of $X_2$, we obtain $\psi(\varphi(y))=y$ for any  $ y \in X_2$.
Since  $\varphi \circ \psi=\id_{X_1}$ and $\psi \circ \varphi=\id_{X_2}$, $\varphi$ is a bijection and $\varphi^{-1}=\psi$. As $X_1$ and $X_2$ are compact Hausdorff spaces and $\varphi$ is a continuous map, $\varphi$ is a homeomorphism. Moreover by (\ref{sss}),  (\ref{ss}) and  $\varphi^{-1}=\psi$, we have 
\[
F(x)=S_x\circ V_{\psi(x)}(F(x)), \quad F \in \coa, \quad x \in X_1.
\]
and 
\[
F(y)=V_y \circ S_{\varphi(y)}(F(y)), \quad F \in \cta, \quad y \in X_2
\]
For any $y \in X_2$, we get  $S_{\varphi(y)}\circ V_{y}=\id_{\ma_1}$ and $V_y \circ S_{\varphi(y)}=\id_{\ma_2}$. Hence $V_y$ is a bijection and $V_y$ is a Jordan $*$-isomorphism such that $V^{-1}_{y}=S_{\varphi(y)}$. Let $\{y_{\lambda}\} \subset X_2$ be a net with $y_{\lambda} \to y_0$. For any $a \in \ma_1$, $\|V_{y_{\lambda}}(a)-V_{y_0}(a)\|_{\ma_2}=\|J(1\otimes a)(y_{\lambda})-J(1 \otimes a)(y_0)\|_{\ma_2} \to 0$ as $J(1 \otimes a) \in \cta$. For any $x \in X_1$, $V^{-1}_{\varphi^{-1}(x)}=S_{\varphi(\varphi^{-1}(x))}=S_x$. Similarly, let $\{x_{\lambda}\} \subset X_1$ be a net with $x_{\lambda} \to x_0$. For any $a \in \ma_2$, $\|V^{-1}_{\varphi^{-1}(x_{\lambda})}(a)-V^{-1}_{\varphi^{-1}(x_0)}(a)\|_{\ma_1}=\|S_{x_{\lambda}}(a)-S_{x_0}(a)\|_{\ma_1}=\|J^{-1}(1 \otimes a)(x_{\lambda})-J^{-1}(1 \otimes a)(x_0)\|_{\ma_1} \to 0 $ because $J^{-1}(1 \otimes a) \in \coa$. We complete the proof.
\end{proof}

\section{Jodan $*$-homomorphisms on $\lxm$ }\label{LC}
 In this section we investigate Jordan $*$-homomorphisms between spaces of Lipschitz maps taking values in a unital $C^*$-algebra. 
Throughout this section,  $X$, $X_1$ and $X_2$ are compact metric spaces. 
We see that $\lxm$ is a non-commutative Banach $*$-algebra with identity. In addition we also define the algebraic tensor product. Let $f \in \Lip(X)$ and $a \in \ma$, where $\Lip(X)$ is the space of all complex valued Lipschitz functions on $X$. We define $f \otimes a(x)=f(x)a$ for any $x \in X$. The algebraic tensor product space of $\Lip(X)$ and $\ma$ will denoted by $\Lip(X) \otimes \ma$. It is easy to see that $\Lip(X) \otimes \ma \subset \lxm$. 
Let $J: \lxom \to \lxtm$ be a Jordan $*$-homomorphism. Because it is not known whether Jordan $*$-homomorphisms in this setting are continuous or not, we assume that $J$ is continuous with $\|\cdot\|_{L}$. If $a$ is an element of a unital Banach algebra, $\sigma(a)$ will denote the spectrum of  $a$ and $r(a)$ its spectral radius. 

\begin{lemma}\label{0}
Then $F$ is an invertible element in $\lxm$ if and only if $F(x)$ is an invertible element in $\ma$ for all $x \in X$. 
\end{lemma}
\begin{proof}
Let $F \in \lxm$ such that $F(x)$ is an invertible element in $\ma$ for all $x \in X$. Given $x_0 \in X$, let $\{x_n\}$ be a sequence of $X$ with $x_n \to x_0$. Then  $F(x_n) \to F(x_0)$. We set $\operatorname{Inv}(\ma)=\{a \in \ma | \ a \text{\ is \ invertible} \}. $ Since the map from $\operatorname{Inv}(\ma)$ onto $\operatorname{Inv}(\ma)$ by $a \mapsto a^{-1}$ is continuous, we get $F(x_n)^{-1} \to F(x_0)^{-1}$. Thus $F^{-1} \in C(X, \ma)$ and $\|F^{-1}\|_{\infty} < \infty$. For any $x, y \in X$, we get 
\begin{multline*}
\|F^{-1}(x)-F^{-1}(y)\|=\|F^{-1}(x)(F(y)-F(x))F^{-1}(y)\|\\\le \|F^{-1}(x)\|\|F(y)-F(x)\|\|F^{-1}(y)\| \le \|F^{-1}\|_{\infty}^2 \|F(y)-F(x)\|.
\end{multline*} 
We have $L(F^{-1})\le \|F^{-1}\|_{\infty}^2 L(F) <\infty$ and consequently, $F^{-1} \in \lxm$. The converse is clear.
\end{proof}

\begin{lemma}\label{1a}
Suppose that a Jordan $*$-homomorphism $J: \lxom \to \lxtm$ is continuous with $\|\cdot\|_{L}$. Then $J$ is continuous with $\|\cdot\|_{\infty}$.
\end{lemma}
\begin{proof}
Let $F \in \lxom$ such that $F(x)^{*}=F(x)$ for all $x \in X_1$.  We firstly show that $r(F)=\|F\|_{\infty}$. 
By Lemma \ref{0}, for any $\la \in \sigma(F)$, there exists $x \in X_1$ such that $F(x)-\la \cdot 1$ is not invertible, and $\la \in \sigma(F(x))$. This implies that $|\la| \le r(F(x))=\|F(x)\| \le \|F\|_{\infty}$. As $\la \in \sigma(F)$, we have $r(F) \le \|F\|_{\infty}$. On the other hand, there is $x_0 \in X_1$ such that $\|F(x_0)\|=\|F\|_{\infty}.$ Hence $\|F\|_{\infty}=\|F(x_0)\|=r(F(x_0))=\sup\{ |\la| | \la \in \sigma(F(x_0))\} \le \sup\{ |\la| | \la \in \sigma(F) \} =r(F)$. We conclude that $\|F\|_{\infty}=r(F)$. Fix $y \in X_2$. For any $n \in \mathbb{N}$, 
\begin{equation}\label{mm}
\|(JF)^n(y)\|^{\frac{1}{n}}=\|J(F^n)(y)\|^{\frac{1}{n}}\le \|J(F^{n})\|_{L}^{\frac{1}{n}} \le \|J\|^{\frac{1}{n}}\|F^n\|_{L}^{\frac{1}{n}}.
\end{equation} 
As $F(x)^{*}=F(x)$ for all $x \in X_1$, we have $(JF)^{*}(y)=JF(y)$.  By (\ref{mm}) we obtain
$\|JF(y)\| \le r(F) =\|F\|_{\infty}$, because $\lim_{n \to \infty}\|(JF)^n(y)\|^{\frac{1}{n}}=r(JF(y))=\|JF(y)\|$ and $\lim_{n \to \infty} \|J\|^{\frac{1}{n}}\|F^n\|^{\frac{1}{n}}_L=1 \cdot r(F)$.  As $y \in X_2$ is arbitrary, we have 
\[
\|JF\|_{\infty}\le \|F\|_{\infty}.
\] For any $F \in \lxom$, we set $F_1=\frac{1}{2}(F + F^*)$ and $F_2=-\frac{1}{2}i (F-F^*)$. Then $F_1$ and $F_2$ are self adjoint elements of $\lxom$ and $F=F_1+iF_2$. In addition, we have $\|F_1\|_{\infty} \le \|F\|_{\infty}$ and $\|F_2\|_{\infty} \le \|F\|_{\infty}$. Thus $\|JF\|_{\infty} \le \|JF_1\|_{\infty}+\|JF_2\|_{\infty} \le \|F_1\|_{\infty}+\|F_2\|_{\infty} \le 2\|F\|_{\infty}$ and $J$ is bounded with $\|\cdot\|_{\infty}.$
\end{proof}
Since $\lxm$ is uniformly dense in $C(X,\ma)$, Lemma \ref{1a} implies that every continuous Jordan $*$-homomorphism from $\lxom$ into $\lxtm$ admits a unique extension to a Jordan $*$-homomorphism $J^{\infty}: \coa \to \cta$. 

\begin{theorem}\label{T8}
Let $J$ be a continuous Jordan $*$- homomorphism from $\lxom$ into $ \lxtm$. Then we define a Jordan $*$-homomorphism $J^{\infty}: \coa \to \cta$ as a unique extension of $J$. Let $\Bb$ be the $C^{*}$-algebra generated by $J^{\infty}(\coa)$. For any $\rdy \in \ps(\Bb)$, let $Y_{\rho}=\{x \in X_2 | \ \rho \circ \delta_{x} \in \ps(\Bb) \}$. Then there exist a Lipschitz map $\varphi_{\rho}: Y_{\rho} \to X_1$ and a Jordan $*$-homomorphism $V_y: \ma_1 \to \ma_2$ for every $y \in Y_{\rho}$ such that 
\[
\pi_{\rho}(JF(y))=\pi_{\rho}(V_y(F(\varphi_{\rho}(y))))
\]
for all $F \in \lxom$ and $y \in \yr$.
\end{theorem}
\begin{proof}
Let $\rdy \in \ps(\Bb)$. By applying Theorem \ref{T1} for a Jordan $*$- homomorphism $J^{\infty}$ from $\coa$ into $\cta$, we only need to show that the continuous map $\varphi_{\rho}: \yr \to X_1$ is Lipschitz continuous. For any $y,z \in \yr$ and $f \in \Lip(X_1)$, we have 
\begin{multline*}
\|\pi_{\rho}(J^{\infty}(f \otimes 1)(y))- \pi_{\rho}(J^{\infty}(f\otimes 1)(z))\|_{B(H_{\rho})} \\ \le \|J(f \otimes 1)(y)-J(f\otimes 1)(z)\|_{\ma_2} \le L(J(f\otimes 1))d(y,z).
\end{multline*}
We get $\pi_{\rho}(J^{\infty}(f \otimes 1)(z))=\pi_{\rho}(V_z(f\otimes 1(\varphi_{\rho}(z))))=\pi_{\rho}(f(\varphi_{\rho}(z))V_{z}(1))=f(\varphi_{\rho}(z))\id_{H_{\rho}}$ for any $f \in \Lip(X_1)$ and $z \in Y_{\rho}$. Fix $y \in Y_{\rho}$. We set $f \in \Lip(X_1)$ by $f(x)=d(x, \varphi_{\rho}(y))$. For any $z \in Y_{\rho}$, we have 
\begin{multline*}
d(\varphi_{\rho}(z), \varphi_{\rho}(y))=\|\pi_{\rho}(J^{\infty}(f \otimes 1)(z))-\pi_{\rho}(J^{\infty}(f \otimes 1)(y))\|_{B(H_{\rho})} \\
\le L(J(f \otimes 1))d(z,y) \le \|J\| \|f\|_{L}d(z,y)\\=\|J\|(\operatorname{diam}(X_2)+1)d(z,y).
\end{multline*}
This implies that $\varphi_{\rho}:\yr \to X_1$ is a Lipschitz map.
\end{proof}

As a corollary of Theorem \ref{T8}, we have the following. This is  a 
slight extension of \cite[Theorem 1]{Oi} because we do not assume that $T(1)=1$ and $X_2$ is connected. 
\begin{cor}
Let $M_1$ and $M_2$ be compact Hausdorff spaces. Suppose that  $T: \Lip(X_1,C(M_1)) \to \Lip(X_2, C(M_2))$ is an algebra homomorphism. Let $\Bb$ be the $C^{*}$-algebra generated by $T^{\infty}(C(X_1, C(M_1)))$. Fix $\delta_{\phi} \circ \delta_{y} \in \ps(\Bb)$, where $y \in X_2$ and $\phi \in M_2$. Then there  exist a Lipschitz map $\varphi_{\delta_{\phi}}: Y_{\delta_{\phi}} \to X_1$  and a  continuous map $\tau: M_2 \to M_1$ such that 
\[
TF(y)(\phi)=F(\varphi_{\delta_{\phi}}(y)))(\tau_y(\phi))
\]
for all $F \in \Lip(X_1,C(M_1))$ and $y \in Y_{\delta_{\phi}}$.
\end{cor}
\begin{proof}
As $\Lip(X_1,C(M_1)) $ and $\Lip(X_2, C(M_2))$ are semi-simple commutative Banach algebras which are symmetric, every algebra homomorphism is an algebra $*$- homomorphism. Moreover any algebra homomorphism between  semi-simple commutative Banach algebras is continuous. Hence, $T$ is continuous with $\|\cdot\|_{L}$. Thus Theorem \ref{T8} implies that there  exist a Lipschitz map $\varphi_{\delta_{\phi}}: Y_{\delta_{\phi}} \to X_1$ and  an algebra $*$- homomorphism $V_y: C(M_1) \to C(M_2)$ for each $y \in Y_{\delta_{\phi}}$ such that 
$\pi_{\delta_{\phi}}(TF(y))=\pi_{\delta_{\phi}}(V_y(F(\varphi_{\delta_{\phi}}(y))))$
for all $F \in \Lip(X_1,C(M_1))$. 
For each $y \in Y_{\delta_{\phi}}$ there is a continuous map $\tau_y$ from $M_2$ into $M_1$ such that $V(a)(\phi)=a(\tau_y(\phi))$ for any $a \in C(M_1)$ and $\phi \in M_1$. Hence we have
\[
TF(y)(\phi)=F(\varphi_{\delta_{\phi}}(y))(\tau_y(\phi))
\]
for all $F \in \Lip(X_1,C(M_1))$ and $y \in Y_{\delta_{\phi}}$.
\end{proof}
\begin{remark}
If $T(1)=1$ and $X_2$ is connected, then the author proved that $T$ preserves constant functions in \cite[Theorem 1]{Oi}. When an algebra $*$-homomorphism $T: \Lip(X_1,C(M_1)) \to \Lip(X_2, C(M_2))$ preserves constant functions, we have $V_x(a)=T(1\otimes a)(x)=T(1\otimes a)(z)=V_z(a)$ for any $x,z \in Y_{\delta_{\phi}}$ and $a \in C(M_1)$. Therefore we define $V: C(M_1) \to C(M_2)$ by $V(a)=V_y(a)$. Then there is a continuous map $\tau$ from $M_2$ into $M_1$ such that $V(a)(\phi)=a(\tau(\phi))$ for any $a \in C(M_1)$ and $\phi \in M_1$. Thus we have 
$TF(y)(\phi)=F(\varphi_{\delta_{\phi}}(y)))(\tau(\phi))$
for all $F \in \Lip(X_1,C(M_1))$ and $y \in Y_{\delta_{\phi}}$. This is the formula given in \cite[Theorem 1]{Oi}.
\end{remark}

We denote the set of all bounded linear operators from $\ma_1$ into $\ma_2$ by $B(\ma_1,\ma_2)$. We say that a map from $X$ into $B(\ma_1,\ma_2)$ by $y \mapsto V_y$ is Lipschitz continuous with respect to the strong operator topology if and only if there is a positive real number $K$ such that $\|V_y(a)-V_{z}(a)\| \le K \|a\| d(y,z)$ for any $a \in \ma_1$ and any $y,z \in X$.

\begin{theorem}\label{T9}
Assume that $\ma_1$ and $\ma_2$ are primitive. Then $J: \lxom \to \lxtm$ is a continuous Jordan $*$-isomorphism if and only if there exist a lipeomorphism $\varphi: X_2 \to X_1$ and  a Jordan $*$-isomorphism $V_y: \ma_1 \to \ma_2$ for each $y \in X_2$, where the map $y \mapsto V_y$ and $x \mapsto V^{-1}_{\varphi^{-1}(x)}$ are Lipschitz continuous with respect to the strong operator topology such that 
\[
JF(y)=V_y(F(\varphi(y)))
\]
for all $F \in \lxom$ and $y \in X_2$.
\end{theorem}
\begin{proof}
Suppose that $J$ is of the form described in the statement. 
Since $\varphi$ is a lipeomorphism, there is a positive number $L(\varphi)$ such that $d(\varphi(y),\varphi(z)) \le L(\varphi)d(y,z)$ for any $y,z \in X_2$.
For any $F \in \lxom$, we have 
\begin{multline*}
\|JF(y)-JF(z)\|_{\ma_2} = \|V_y(F(\varphi(y)))-V_z(F(\varphi(z)))\|_{\ma_2}\\
\le \|V_y(F(\varphi(y)))-V_y(F(\varphi(z)))\|_{\ma_2} + \|V_y(F(\varphi(z)))-V_z(F(\varphi(z)))\|_{\ma_2}\\
\le \|F(\varphi(y))-F(\varphi(z))\|_{\ma_1}+K\|F(\varphi(z))\|d(y,z)\\
\le L(F)L(\varphi)d(y,z)+K\|F\|_{\infty}d(y,z)=(L(F)L(\varphi)+K\|F\|_{\infty})d(y,z).
\end{multline*}
This shows that $JF \in \lxtm$. In addition we have $\|JF\|_{L} \le (1+\max \{ L(\varphi), K\}) \|F\|_{L}$. Applying similar arguments, we obtain that $J^{-1}F \in \lxom$ for all $F \in \lxtm$. Thus  $J$ is a continuous Jordan $*$-isomorphism from $\lxom$ onto $\lxtm$. 

Here, we prove the converse.  By Theorem \ref{btry}, it remains only to show that $\varphi$ is a lipeomorphism and the maps $y \mapsto V_y$ and $x \mapsto V_{\varphi^{-1}(x)}^{-1}$ are Lipschitz continuous with respect to the strong operator topology. We define $\varphi=\varphi_{\rho}$, where $\rho \in \ps(\ma_2)$ such that $\pi_{\rho}$ is a faithful non-zero irreducible representation of $\ma_2$. Thus Theorem \ref{T8} asserts $\varphi$ is a Lipschitz map and $JF(y)=V_y(F(\varphi(y)))$ for any $F \in \lxom$. With a 
similar argument as for $J^{-1}: \lxtm \to \lxom$, there exists a Lipschitz map $\psi: X_1 \to X_2$ and a Jordan $*$-isomorphism $S_x : \ma_2 \to \ma_1$ for each $x \in X_1$ such that $(J^{-1}F)(x)=S_{x}(F(\psi(x)))$. Since $\Lip(X_i, \ma_i)$ separates the points of $X_i$, we get $\varphi \circ \psi=\id_{X_1}$ and $ \psi \circ \varphi=\id_{X_2}$ (This is discussed in more 
detail in the proof of Theorem \ref{btry}).  Thus  $\varphi^{-1}=\psi$ . We conclude that $\varphi$ is a lipeomorphism. For any $a \in \ma_1$, 
\begin{multline*}
\|V_{y}(a)-V_{z}(a)\|=\|J(1 \otimes a)(y)-J(1 \otimes a)(z)\| \\ \le L(J(1 \otimes a))d(y,z) \le \|J\|\|a\|d(y,z)
\end{multline*}
for each $y,z \in X_2$. Since $V^{-1}_{\varphi^{-1}(x)}=S_x$ for any $x \in X_1$,  we have 
\begin{multline*}
\|V^{-1}_{\varphi^{-1}(x)}(a)-V^{-1}_{\varphi^{-1}(z)}(a)\|=\|S_x(a)-S_z(a)\|=\|J^{-1}(1 \otimes a)(x)-J^{-1}(1 \otimes a)(z)\|\\
\le L(J^{-1}(1 \otimes a))d(x,z) \le \|J^{-1}\|\|a\|d(x,z)
\end{multline*}
for any $a \in \ma_1$ and any $x,z \in X_1$. Hence we finish the proof. 
\end{proof}

\begin{remark}
Since $M_{n}(\mathbb{C})$ and $B(H)$ are primitive $C^{*}$-algebras, Theorem \ref{T9} implies that we characterize continuous Jordan $*$-isomorphisms on $\Lip(X,M_n(\mathbb{C}))$ and $\Lip(X,B(H))$ as corollaries. We do not assume that a Jordan $*$-isomorphism preserves constant functions. Thus Theorem \ref{T9} is a generalization of \cite[Theorem 1.1]{BJm} and \cite[Theorem 1.2]{BJb}. 
\end{remark}

\subsection*{Acknowledgments}
This work was supported by JSPS KAKENHI Grant Numbers JP21K13804.

\end{document}